\documentclass[11pt,english]{article}
\usepackage[T1]{fontenc}
\usepackage[latin9]{inputenc}
\usepackage{geometry}
\usepackage{verbatim}
\usepackage{calc}
\usepackage{amsthm}
\usepackage{amsmath}
\usepackage{amssymb}
\usepackage{graphicx}
\usepackage{esint}

\makeatletter
\theoremstyle{plain}
\newtheorem{thm}{\protect\theoremname}
  \theoremstyle{definition}
  \newtheorem{defn}[thm]{\protect\definitionname}
  \theoremstyle{definition}
  \newtheorem{example}[thm]{\protect\examplename}
  \theoremstyle{remark}
  
  \theoremstyle{plain}
  \newtheorem{lemma}[thm]{\protect\lemmaname}

\theoremstyle{plain}
\newtheorem{prop}{\protect\propositionname}

\theoremstyle{plain}
\newtheorem{cor}{\protect\corname}

  \theoremstyle{definition}
  \newtheorem{algorithm}[thm]{\protect\algname}

  \theoremstyle{definition}
  \newtheorem{computation}[thm]{\protect\computationname}

\makeatother

\usepackage{babel}
  \providecommand{\definitionname}{Definition}
  \providecommand{\examplename}{Example}
  \providecommand{\lemmaname}{Lemma}
  \providecommand{\remarkname}{Remark}
\providecommand{\theoremname}{Theorem}

\providecommand{\propositionname}{Proposition}
\providecommand{\corname}{Corollary}
\providecommand{\algname}{Algorithm}
\providecommand{\computationname}{Computation}

\DeclareMathOperator{\Jac}{Jac}

\DeclareMathOperator{\MLdegree}{MLdegree}

\DeclareMathOperator{\minors}{minors}
\DeclareMathOperator{\codim}{codim}

\newcommand{\cPu}{ {\cal{P}}_X(u)}
\newcommand{\cBu}{ {\cal{B}}_X(u)}
\newcommand{\lup}{l_u(p)}
\newcommand{\luP}{l_u'(p)}
\newcommand{\luB}{l_u'(b)}

\newcommand{\N}{\mathbb{N}}

\newcommand{\PP}{\mathbb{P}}

\begin{document}
%

\title{Maximum Likelihood for Dual Varieties}

\author{Jose Israel Rodriguez\thanks{
Department of Mathematics,
University of California at Berkeley,
Berkeley, CA 94720; 
  {\tt jo.ro@berkeley.edu.}}}

\date{19 May 2014}

\maketitle
\begin{abstract}
Maximum likelihood estimation (MLE) is a fundamental computational problem in statistics.  
In this paper, MLE for  statistical models with discrete data  is studied from an algebraic statistics viewpoint. 
A reformulation of the MLE problem in terms of dual varieties and conormal varieties will be given. 
With this description,  the dual likelihood equations and the dual MLE problem are defined.
We show that solving the dual MLE problem yields solutions to the MLE problem, so  we can solve the  MLE problem without ever determining  the defining equations of the~model.  
\end{abstract}

\section{Introduction}

Maximum likelihood estimation (MLE) is a fundamental problem in statistics  that has been extensively studied from an algebraic viewpoint \cite{CHKS06,DR13,DSS09,HRS13,HKS05,Huh13}.
We continue to follow  an algebraic approach to MLE in this paper considering  statistical  models for discrete data in the probability simplex  as irreducible varieties $X$ in complex projective space~$\mathbb{P}^n$.

An {\em algebraic statistical  model $X$} in $\mathbb{P}^n$  will be  defined by  the vanishing of homogeneous polynomials in the unknowns $p_0,p_1,\dots,p_n$. 
We assume that $X$ is an irreducible generically reduced variety.  
When the coordinates $p_0,p_1,\dots,p_n$ of a point $p$ in  $X$  are positive and sum to one, we interpret $p$  as a probability distribution, where the probability of observing event $i$ is $p_i$.   
We let $u=(u_0,u_1,\dots,u_n)\in(\mathbb{C}^{*})^{n+1}$ 
be a vector of length $n+1$ called \textit{data}.  
When each entry $u_i$ of the vector is a positive integer we interpret $u_i$ as the number of observations of event $i$.  
We use the notation  
$$u_+:=u_0+\cdots+u_n\text{ and }p_+:=p_0+\cdots +p_n,$$
 always assuming $u_+\neq 0$.

The likelihood function for  $u$ is defined as 
$$\lup:=p_0^{u_0}p_1^{u_1}\cdots p_n^{u_n}/p_+^{u_+}.$$
When $u$ and $p$ are interpreted as data and a probability distribution respectively, 
the likelihood of observing $u$ with respect to the distribution $p$ is $\lup$ divided by a multinomial coefficient depending only on $u$.
 
For fixed data $u$, to determine local maxima of $\lup$ on a statistical model and give a solution to the MLE problem, we determine all complex critical points of $\lup$ restricted to $X$.  Of these critical points, we find the one with positive coordinates and  greatest likelihood to determine the maximum likelihood estimator $\hat p$. 
The  \textit{(algebraic)  maximum likelihood estimation problem}  is solved by determining all critical points of $\lup$ on $X$ and maximizing $\lup$ on this set. 

To find the complex critical points, we determine when the gradient of $\lup$  is orthogonal to the tangent space of $X$ at $p$. So the set of critical points is
$$\{p\in X_{reg}\text{ such that }\nabla \lup\perp T_p X\}.$$
 The gradient of the likelihood function  equals
 $\left[\frac{u_0}{p_0}-\frac{u_+}{p_+},\frac{u_1}{p_1}-\frac{u_+}{p_+},\dots,\frac{u_n}{p_n}-\frac{u_+}{p_+}\right]$,
 up to scaling by
 $\lup/p_+^{u_+}$.
So the critical points of $\lup$ are $p\in X_{reg}$ such that 
$$\left[\frac{u_0}{p_0}-\frac{u_+}{p_+},\frac{u_1}{p_1}-\frac{u_+}{p_+},\dots,\frac{u_n}{p_n}-\frac{u_+}{p_+}\right]\perp T_p(X),$$
implicitly forcing the condition $p_0p_1\cdots p_n(p_0+\cdots +p_n)\neq 0$.

\begin{defn} \label{def:mldegree} 
Given an algebraic statistical model $X$ in $\mathbb{P}^n$, the {\em maximum likelihood degree} (ML degree) of $X$ is the number of critical points of $\lup$ restricted to $X$ for generic choices of data $u$,
$$\MLdegree (X)=\# \left\{p\in X : \nabla\lup\perp T_p(X)\right\}.$$
\end{defn}

The main result of this paper is to give a formulation that relates maximum likelihood estimation to a conormal variety derived from $X$ [Theorem \ref{thm:projMLD}]. With this  perspective, we use the dual likelihood equations [Theorem \ref{dualEq}] to solve the MLE problem for $X$  when  only given the defining equations of its dual variety $X^{*}$.

The computations in this paper were done using \texttt{Bertini}~\cite{Bertini}
 and  \texttt{Macaulay2} \cite{M2}.

\section{Maximum Likelihood Estimation}

In this section, we consider an algebraic statistical model $X$  in $\mathbb{P}^n$ and will define $X'$ to be an embedding of $X$ in $\mathbb{P}^{n+1}$.
We will present our first result in Theorem \ref{thm:projMLD} 
giving a formulation of the MLE problem in terms of 
conormal varieties and dual varieties. 
In Corollary \ref{bijectionPB} we  give a bijection between critical points of the likelihood function on two different varieties. 
In Corollary \ref{cor:mlequations} we  give equations to solve the MLE problem if we have equations that define a conormal variety. 
We will also recall how to compute conormal varieties and dual varieties of $X$ and $X'$.  

Let $X\subset \mathbb{P}^n$  be a codimension $c$  algebraic statistical model defined by homogenous polynomials $f_1,f_2,\dots,f_k$. 
We let $\Jac(X)$ denote the $k\times (n+1)$ matrix of partial derivatives of $f_1,\dots,f_k$ with respect to $p_0,\dots,p_n$, and we say this is the {\em Jacobian of} $X$.

To keep track of the sum of the coordinates $p_0,p_1,\dots,p_n$ we introduce the coordinate $p_s$ and a hyperplane in $\mathbb{P}^{n+1}$  defined by the vanishing of the polynomial 
\begin{equation}\label{Heq}
H(p):=-p_0-p_1+\cdots-p_n+p_s.
\end{equation} 

If $X$ is defined by $f_1,\dots,f_k$
 then 
 $X'$ in the coordinates $p_0,p_1,\dots,p_n,p_s$ is defined by the vanishing of $f_1,\dots,f_k$ and $H$.
With this definition, we get the following proposition.

\begin{prop}\label{jacX}
If $X$ is defined by the homogeneous polynomials $f_1,f_2,\dots,f_k$ then the Jacobian of $X'$ is given by the $(k+1)\times (n+2)$-matrix 
\[
\Jac(X')=\left[\begin{array}{cc}
\begin{array}{cccc}
-1 & -1 & \cdots & -1\end{array} & 1\\
\Jac\left(X\right) & \begin{array}{c}
0\\
\vdots\\
0
\end{array}
\end{array}\right].
\]
\end{prop}

The important fact about the construction of $X'$ 
is that there is a bijection between the critical points of the function $\lup$ on $X$ and the critical points of the monomial
\[
\luP:=p_0^{u_0}p_1^{u_1}\cdots p_n^{u_n}p_s^{-u_+}\text{ on } X' 
\]
given by Lemma \ref{xprimeBijection}.

By a slight abuse of notation
 the ``$p$'' in  $\lup$  and the ``$p$'' in $\luP$ represent two different things. 
 The first $p$ represents a point $[p_0:p_1:\dots:p_n]\in X$, while the second represents a point $[p_0:p_1:\dots:p_n:p_s]\in X'$.
 
\begin{lemma}\label{xprimeBijection} 
There is a bijection between the critical points of the function $\lup$ on $X$ and the critical points of 
$\luP$~
 on $X'$. 
 Under this bijection,
$[p_0:p_1:\dots :p_n]\in\PP^n$ is a critical point of $\lup$ on $X$ if and only if  
  $[p_0:p_1:\dots :p_n:p_s]\in\PP^{n+1}$ is a critical point of $\luP$ on $X'$.
 \end{lemma}
\begin{proof}
To prove this we need to show that  
\[
[p_0:\dots:p_n:p_s]\in X'_{reg} \text{ satisfies } \nabla \luP\perp T_pX'
\]
if and only if 
\[
[p_0:\dots:p_n]\in X_{reg} \text{ satisfies } \nabla \lup\perp T_pX.
\]
By Proposition \ref{jacX}, it follows $[p_0:\dots:p_n:p_s]\in X'_{reg}$ if and only if 
$[p_0:\dots:p_n]\in X_{reg}$, so it remains to show that 
$\nabla \luP\perp T_pX'$ if and only if $\nabla \lup\perp T_pX$.
So  we need to show that 
$\nabla \luP$ being in the row space of $\Jac(X')$  implies that
$\nabla \lup$ is in the row space of $\Jac(X)$ and vice versa. 
To see this, observe that  

\[
\left[\begin{array}{c}
\nabla\luP\\
\Jac\left(X'\right)
\end{array}\right]
\left[\begin{array}{cccc}
1 &  &  & \\
 & \ddots &  & \\
 &  & 1 & \\
1 & \cdots & 1 & 1
\end{array}\right]=
\]
\[
\left[\begin{array}{cc}
\begin{array}{cccc}
\frac{u_{0}}{p_{0}}-\frac{u_{+}}{p_{s}} & \frac{u_{1}}{p_{1}}-\frac{u_{+}}{p_{s}} & \cdots & \frac{u_{n}}{p_{n}}-\frac{u_{+}}{p_{s}}\\
0 & 0 & \cdots & 0
\end{array} & \begin{array}{c}
-\frac{u_{+}}{p_{s}}\\
1
\end{array}\\
\Jac\left(X\right) & \begin{array}{c}
0\\
\vdots\\
0
\end{array}
\end{array}\right]
\]
Since $p_s=p_+$, we have completed the proof because the top row in the matrix above is 
$\left[\nabla\lup,-\frac{u_+}{p_+}\right]$.\end{proof}

The {\em conormal variety of} $X$  is defined to be  the Zariski closure in $\mathbb{P}^n\times \mathbb{P}^n$ of the set 
$$N_X:=\overline{\left\{ (p,q):q\perp T_p X_{reg}\right\}}.$$ 
To determine the defining equations of $N_X$, we let $M$ denote a   $(k+1)\times (n+1)$ matrix that is an extended Jacobian whose top row is $[q_0,q_1,\dots,q_n]$ and whose bottom rows are $\Jac(X)$.
The defining equations of the conormal variety can be computed by taking the ideal generated by $f_1,\dots,f_k$ and the $(c+1)\times (c+1)$-minors of  $M$ and saturating by the $c\times c$-minors of $\Jac(X)$. 

The \textit{dual variety} $X^{*}$ is the projection of the conormal variety $N_X$  to the dual projective space $\mathbb{P}^n$ associated to the $q$-coordinates. To compute the equations  of the dual variety, one eliminates the unknowns $p_0,p_1,\dots,p_n$ from the equations defining $N_X$. 
For additional information on computing conormal varieties and dual varieties see \cite{RS13}.

Since $X'$ is contained in a hyperplane defined by $H$, 
 the dual variety of $X'$ is known to be a  cone of $X^{*}$ over the point $h=[-1:-1:\dots:-1:1]$\cite{Ein86} (Proposition 1.1). 
 So $X'^{*}$  in $\mathbb{P}^{n+1}$ is given by  
 $$
 \begin{array}{l}
 X'^{*}=\\
\,\, \overline{
 \left\{[q_0-b_s:q_1-b_s:\dots:q_n-b_s:b_s] :  [q_0:\dots:q_n]\in X^{*}
 \right\} }
\end{array}
 $$

It is easy to go between the coordinates of $X$ and coordinates of $X'$ because there there is birational map between these two varieties.  
But there does not have to be a birational map between $X^{*}$ and $X'^{*}$. 
For this reason, the coordinates of the former are in $q_0,\dots,q_n$, and the coordinates of the latter are in $b_0,\dots,b_n,b_s$. 
Our notation is to let $q$ denote a point $[q_0:q_1:\dots :q_n]\in X^{*}$ and to  let $b$ denote a point~$[b_0:b_1:\dots:b_n:b_s]\in X'^{*}$.

The next proposition shows that  given the defining equations of $X^{*}$ in the unknowns $q_0,\dots,q_n$,  we can determine the defining equations of $X'^{*}$ in the unknowns $b_0,\dots,$ $b_n,b_s$ using the relations
\begin{equation}\label{qbRelate}
q_0=b_0+b_s,q_1=b_1+b_s,\dots,q_n=b_n+b_s.
\end{equation}
Meaning, if $g(q_0,q_1,\dots,q_n)$ vanishes on $X^{*}$, then 
$g(b_0+b_s,b_1+b_s,\dots,b_n+b_s)$ vanishes on $X'^{*}$. 
Moreover,  given the Jacobian of $X^{*}$, we can easily determine the Jacobian of $X'^{*}$ as well using the relations in \eqref{qbRelate}.

\begin{prop}\label{jacXDual}
If $g_1(q),\dots,g_l(q)$  define the variety $X^{*}\subset\PP^{n}$ in  coordinates $q_0,q_1,\dots,q_n$, then 
the defining equations of $X'^{*}$ in coordinates $b_0,b_1,\dots,b_n,b_s$ are 
\[
\begin{array}{c}
g_1(b_0+b_s,b_1+b_s,\dots,b_n+b_s)=0\\
\vdots \\
g_l(b_0+b_s,b_1+b_s,\dots,b_n+b_s)=0.
\end{array}
\] 
Moreover, the Jacobian of $X'^{*}$ is given by 
\[
\Jac \left(X'^{*}\right)=
\left.\Jac\left(X^{*}\right)\right|_{\left(b_{0}+b_{s},\dots,b_{n}+b_{s}\right)}
\left[\begin{array}{ccccc}
1 &  &  &  & 1\\
 & 1 &  &  & \vdots\\
 &  & \ddots &  & 1\\
 &  &  & 1 & 1
\end{array}\right].
\]

\end{prop}
\begin{proof}
The first part of proposition follows immediately from the relations in \eqref{qbRelate}.
By $$\left.\Jac\left(X^{*}\right)\right|_{\left(b_{0}+b_{s},\dots,b_{n}+b_{s}\right)}$$ we mean evaluate the Jacobian of $X^{*}$ at 
${\left(b_{0}+b_{s},\dots b_{n}+b_{s}\right)}$. 
Since the defining equations of $X'^{*}$ are gotten by evaluating each $g_i(q)$ at 
${\left(b_{0}+b_{s},\dots b_{n}+b_{s}\right)}$, it follows by the chain rule that $\Jac(X'^{*})$ equals the desired matrix product. \end{proof}

\begin{example}\label{example:babyDualAndConormal}
Consider  $X$ in $\PP^3$, a variety  defined by 
$$f=2p_0p_1p_2+p_1^2p_2+p_1p_2^2-p_0^2p_{12}+p_1p_2p_{12}.$$
The Jacobian of $X$ and the defining polynomial $g(q)$ of the dual variety $X^{*}$ are 
\[
\begin{array}{r}
\Jac(X)=[2p_1p_2-2p_0p_{12},2p_0p_2+2p_1p_2+p_2^2+p_2p_{12},\\
 2p_0p_1+p_1^2+2p_1p_2+p_1p_{12},-p_0^2+p_1p_2]
\end{array}
\]
and
\[
\begin{array}{r}
g(q)=q_0^4-8q_0^2q_1q_2+16q_1^2q_2^2-8q_0^3q_{12}+\\16q_0^2q_1q_{12}+
16q_0^2q_2       q_{12}-32q_0q_1q_2q_{12}.
\end{array}
\]
The variety  $X'$ is defined by the two equations,
$$f(p)=0\text{ and }p_s=p_0+p_1+\cdots+p_n,$$
but the dual variety $X'^{*}$ is defined  by one equation
\[
\begin{array}{c}
g(b_0+b_s,b_1+b_s,b_2+b_s,b_{12}+b_s)=\\
(b_0+b_s)^4-8(b_0+b_s)^2(b_1+b_s)(b_2+b_s)+\\
16(b_1+b_s)^2(b_2+b_s)^2-8(b_0+b_s)^3(b_{12}+b_s)+\\
16(b_0+b_s)^2(b_1+b_s)(b_{12}+b_s)+\\
16(b_0+b_s)^2(b_2+b_s)       (b_{12}+b_s)\\
-32(b_0+b_s)(b_1+b_s)(b_2+b_s)(b_{12}+b_s).
\end{array}
\]
The Jacobian of $X^{*}$ is 
$$
\left[
\begin{array}{c}
4q_0^3-16q_0q_1q_2-32 q_{12}(\frac{3}{4}q_0^2-q_0q_1-q_0q_2+q_1q_2)\\
-8q_0^2q_2+32q_1q_2^2+16q_0^2q_{12}-32q_0q_2q_{12}\\
-8q_0^2q_1+32q_1^2q_2+16q_0^2q_{12}-32q_0q_1q_{12}\\
-8q_0^3+16q_0^2q_1+16q_0^2q_2-32q_0q_1q_2
\end{array}
\right]^T.
$$
We get the Jacobian of $X'^{*}$ by evaluating $\Jac(X^{*})$ at \break
$\left(b_{0}+b_{s},\dots b_{n}+b_{s}\right)$ and multiplying the evaluated $\Jac(X^{*})$ 
on the right by 
the matrix 
\[
\left[\begin{array}{ccccc}
1 &  &  &  & 1\\
 & 1 &  &  & 1\\
 &  & 1 &  & 1\\
 &  &  & 1 & 1
\end{array}\right].
\]

\end{example}

Now we are ready to state our first result.

\begin{thm}\label{thm:projMLD}
Fix an algebraic statistical model $X$.   A point 
$$\big([p_0:p_1:\dots:p_n:p_s],[b_0:b_1:\dots:b_n:b_s]\big)\in N_{X'}$$
satisfies the relation 
\[
\left[p_0b_0:p_1b_1:\dots:p_nb_n:p_{s}b_s\right]=\left[u_0:u_1:\dots:u_n:-u_{+}\right]
\]
if and only if 
$[p_0:p_1:\dots :p_n:p_s]$ is a critical point of $\luP=p_0^{u_0}p_1^{u_1}\cdots p_n^{u_n}p_s^{-u_+}$~on~$X'$. 


\end{thm}

\begin{proof}
To determine  critical points of $\luP$ on $X'$ we
find when 
$$\nabla \luP=\left[\partial l_{u}'/\partial p_{0}:\dots:\partial l_{u}'/\partial p_{s}\right]$$
is orthogonal to the tangent space of $X'$ at the point $p$. 
This
is the same as determining~when 
$$\big([p_0:p_1:\dots:p_s],\nabla \luP\big)\in N_{X'}.$$
As a point in projective space, we have 
$$\nabla \luP=\left[\frac{u_{0}}{p_{0}},\dots,\frac{u_{n}}{p_{n}},-\frac{u_{+}}{p_{s}}\right]$$
whenever $p_{0}p_{1}\cdots p_{s}\neq0$. So we immediately have that a
critical point of $\luP$  satisfies the desired relations when we take the coordinate-wise product of $[p_0:p_1:\dots:p_s]$ and~$\nabla \luP$. \end{proof}

With Lemma \ref{xprimeBijection}, Theorem \ref{thm:projMLD} says that if 
$[p,b]\in N_{X'}$ and the coordinate-wise product of $p$ and $b$ is
\begin{equation}\label{pbu}
\left[p_0b_0:\dots:p_nb_n:p_{s}b_s\right]=\left[u_0:\dots:u_n:-u_{+}\right],
\end{equation} 
then $[p_0:\dots :p_n]$ is a critical point of $\lup$ on $X$.

\begin{defn} \label{def:ll} 
The $\emph{likelihood locus of}$ $X$ for the data $u$ is defined as the set of points in $N_{X'}$ satisfying the relations in  \eqref{pbu}, notated $L_X(u)$. 
We define $\cPu$ and $\cBu$ to be 
$$\cPu:=\{p:(p,b)\in L_X(u) \}
\text{ and } \cBu:=\{b:(p,b)\in L_X(u) \}.$$
\end{defn}

  For additional clarification, note that points in $L_X(u)$ are contained in the conormal variety $N_{X'}\subset\PP^{n+1}\times\PP^{n+1}$.
These points  are expressed as 
$$(p,b)=\big([p_0:p_1:\dots:p_s],[b_0:b_1:\dots:b_s]\big)\in L_X(u).$$
In regards to ML degree, we have for generic choices of $u$ 
$$\MLdegree(X)=\#L_X(u)=\#\cPu=\#\cBu.$$

There are two corollaries to Theorem \ref{thm:projMLD}.
The first corollary gives a bijection between critical points of $\luP$ on $X'$ and critical points of $\luB$ on $X'^{*}$. 
The second corollary gives equations to determine critical points of $\luP$ on $X'$.

\begin{cor}\label{bijectionPB}
There is a bijection between critical points of $\luP$ on $X'$ and critical points of $\luB$ on $X'^{*}$ given by 
$$\left[p_0b_0:p_1b_1:\dots:p_nb_n:p_{s}b_s\right]=\left[u_0:u_1:\dots:u_n:-u_{+}\right].$$
Moreover, the product $\luP\luB$ remains constant over the set of critical points. 
\end{cor}

\begin{proof}
The first part follows by noticing that the relation forces us to have 
$$[p_0:p_1:\dots:p_s]=[u_0/b_0:u_1/b_1:\dots:-u_+/b_s]$$ 
which is also the gradient of  $\luB$.
The second part follows as 
$$\luP\luB=u_0^{u_0}u_1^{u_1}\cdots u_n^{u_n}(-u_+)^{-u_{+}}.$$
\end{proof}

When $u_0,\dots,u_n$ are positive integers, the bijection in  Corollary \ref{bijectionPB} 
pairs positive critical points of $\luP$ ordered by increasing  likelihood with positive critical points of $\luB$ ordered by decreasing likelihood!  

\begin{example}\label{rcmodel}
We will compute the ML degree of  $X$ in Example \ref{example:babyDualAndConormal} to be $3$.
We fix the data vector $\left(u_0,u_1,u_2,u_{12}\right)=\frac{1}{40}\left(2,13,5,20\right)$, and determine the points of $L_X(u)$
\[
\begin{array}{ccccc}
p_0&p_1&p_2&p_{12}&p_s\\
 .167493 & .242186 & .0532836 & .537037 &1\\ 
-.485608 & .632011& .35886& .494736 &1\\ 
-2.58189 & 5.56009& 6.19312& -8.17133 &1\\
b_0&b_1&b_2&b_{12}&b_s\\
.29852 & 1.34194 & 2.34594 & .931035&-1\\
-.102964 & .514232 & .348325& 1.01064&-1\\
-.0193657 & .0584523 & .0201837& -.0611895&-1.
\end{array}
\]
The eliminants for $p_0,p_1,p_2$, and $p_{12}$ are
\[
\begin{array}{c}
(100p_0^3+290p_0^2+74p_0-21),\\
(62700p_1^3-403430p_1^2+314358p_1-53361),\\
(1900p_2^3-12550p_2^2+4886p_2-225),\\
(62700p_{12}^3+447650p_{12}^2-511962p_{12}+136125).
\end{array}
\] 
The eliminants for $b_0,b_1,b_2,b_{12}$ of $L_X(u)$ are 
\[
\begin{array}{c}
(1680b_0^3-296b_0^2-58b_0-1),\\
        (34151040b_1^3-65386464b_1^2+27271868b_1-1377519),\\
        (28800b_2^3-78176b_2^2+25100b_2-475),\\
        (272250b_{12}^3-511962b_{12}^2+223825b_{12}+15675).
\end{array}
\]        
\end{example}

Note that  we are \textit{not} saying that the ML degree of $X$ equals the ML degree of $X^{*}$.
In general, 
$$\MLdegree(X)\neq\MLdegree(X^{*}).$$ 
The reason why equality fails is because 
$$b_0+b_1+\cdots+b_n-b_s$$ 
does not  vanish on $X'^{*}$. So there is no analogue of Lemma \ref{xprimeBijection} involving $X'^{*}$ and $X^{*}$. 
In terms of  previous literature,  one should think  of Corollary \ref{bijectionPB}
as a generalization of Theorem 2 of \cite{DR13}.
  
\begin{cor}\label{cor:mlequations}
Fix a point $[p,b]$ of  $N_{X'}$ such that $p_sb_s\neq 0$. The following are equivalent:
\begin{enumerate}
\item\label{mlequations1}
The point $[p,b]$  is in $L_X(u)$.
\item\label{mlequations2}
For   $i=0,1,2,\dots,n$, the point $[p,b]$ satisfies 
\[
 u_ip_sb_s=-u_+p_ib_i.
\]
\item\label{mlequations3}
There  exists $[q_0:\dots:q_n]\in X^{*}$ such that
for  $i=0,1,\dots,n$
\[
u_ip_sb_s=-u_+p_i(q_i-b_s). 
\]
\end{enumerate}
\end{cor}
\begin{proof}
It is immediate that parts \ref{mlequations1} and \ref{mlequations2} are equivalent.  
To see \ref{mlequations2} and \ref{mlequations3} are equivalent, 
 recall $q_i=b_i+b_s$ for $i=0,1,\dots,n$, from the definition of $X'^{*}$.
\end{proof}
A consequence of these equations is that it removes the need for saturation by  $p_0p_1\cdots p_n$ with  Grobner basis computations that involve the likelihood equations  whenever the $u_i$ are nonzero. In addition, if we restrict to the affine charts defined by  $p_s=1$ and $b_s=-u_+$, then the condition $p_sb_s\neq0$ is immediately satisfied.

\section{Dual Likelihood Equations}
In this section we will define a system of  equations whose solutions are precisely
$${\cBu}=\left\{b: (p,b)\in L_X(u)\right\}.$$
Once we know the set  $\cBu$, we  determine the critical points of $\lup=p_0^{u_0}\cdots p_n^{u_n}/p_+^{u_+}$ on $X$ 
using Lemma \ref{xprimeBijection} and Corollary \ref{cor:mlequations}.
Concretely, if 
$b\in\cBu$ then  
$$[p_0:\dots:p_n]=[{u_0}/{b_0}:\dots:{u_n}/{b_n}]$$ is a critical point of $\lup$ on $X$. 
For this reason we make the following definition. 

\begin{defn}
The \textit{dual maximum likelihood estimation problem} for the algebraic statistical model $X$ and data $u$ is to determine ${\cBu}$, the set of  critical points of $\luB$~on~$X'^{*}$.  
\end{defn}

By Corollary \ref{bijectionPB},  we find the critical points of $\luB=b_0^{u_0}b_1^{u_1}\cdots b_n^{u_n}b_s^{-u_+}$ on $X'^{*}$
to determine the set $\cBu$. That is, we determine the points $b\in X'^{*}$ such that the~gradient
$$\nabla \luB=\left[\frac{u_0}{b_0}:\frac{u_1}{b_1}:\dots:\frac{u_n}{b_n}:\frac{-u_+}{b_s}\right]$$ 
is orthogonal to the tangent space of $X'^{*}$ at $b$. 

If $X^{*}$ has codimension $c$, which also means $X'^{*}$ has codimension $c$,  then the {\em dual likelihood equations} are obtained by taking the sum of ideals generated by 
\begin{itemize}
\item
the polynomials defining $X'^{*}$, and
\item 
the $(c+1)\times (c+1)$ minors of an extended Jacobian multiplied by a diagonal matrix with entries $b_0,b_1,\dots,b_n,b_s$,
\begin{equation}\label{gjb}
\left[\begin{array}{c}
\nabla l_{u}\left(b\right)\\
\Jac\left(X'^{*}\right)
\end{array}\right]
\left[\begin{array}{ccc}
b_{0}\\
 & \ddots\\
 &  & b_{s}
\end{array}\right],
\end{equation}
\end{itemize}
and saturating by the product of two ideals,
\begin{itemize}
\item
the principal ideal generated by $b_0b_1\cdots b_nb_s$, and 
\item the ideal generated by the $c\times c$-minors of $\Jac(X'^{*}).$
\end{itemize}
This gives us a formulation of the  dual likelihood equations. Now we make some simplifications to these equations to get Theorem \ref{dualEq}.

By Euler's relations of partial derivatives the columns of the matrix product in \eqref{gjb} are linearly dependent. Indeed the columns sum to zero, so we may drop the last column of the  product without changing the rank.  

By Proposition \ref{jacXDual}, if  $g_1(q),\dots,g_l(q)$ define the variety $X^{*}$, then the defining equations of $X'^{*}$ are 
$$
\begin{array}{c}
g_1(b_0+b_s,b_1+b_s,\dots,b_n+b_s)=0\\
\vdots \\
g_l(b_0+b_s,b_1+b_s,\dots,b_n+b_s)=0.
\end{array}
$$
  and the Jacobian of $X'^{*}$ is 
\[
\Jac \left(X'^{*}\right)=
\left.\Jac\left(X^{*}\right)\right|_{\left(b_{0}+b_{s},\dots b_{n}+b_{s}\right)}
\left[\begin{array}{ccccc}
1 &  &  &  & 1\\
 & 1 &  &  & \vdots\\
 &  & \ddots &  & 1\\
 &  &  & 1 & 1
\end{array}\right].
\]
Since the last column of $\Jac(X'^{*})$ is the sum of the first columns, it follows the dual likelihood equations can be reformulated by the next theorem. 

\begin{thm}\label{dualEq}

Let $g_1(q),\dots,g_l(q)$ define $X^{*}\subset\mathbb{P}^n$ with codimension $c$. Then, $\cBu$ is variety of the ideal calculated by taking the sum of the ideals generated by 
\begin{itemize}
\item
$g_1(b_0+b_s,\dots,b_n+b_s),
\dots,
g_l(b_0+b_s,\dots,b_n+b_s)
$
 and
 \item the $(c+1)\times(c+1)$ minors of 
 \begin{equation}\label{rankDefDual}
\left[\begin{array}{c}
\frac{u_{0}}{b_{0}}\quad \frac{u_{1}}{b_{1}}\quad  \dots \quad \frac{u_{n-1}}{b_{n-1}}\quad \frac{u_{n}}{b_{n}}\\
\\
\quad\quad\quad  \left.\Jac\left(X^{*}\right)\right|_{\left(b_{0}+b_{s},\dots b_{n}+b_{s}\right)}
\end{array}\right]
\left[\begin{array}{ccc}
b_{0}\\
 & \ddots\\
 &  & b_{n}
\end{array}\right],
\end{equation}
\end{itemize}
and saturating by 
the product of two ideals,
\begin{itemize}
\item
the principal ideal generated by $b_0b_1\cdots b_nb_s$, and 
\item the ideal of $c\times c$-minors of $\left.\Jac\left(X^{*}\right)\right|_{\left(b_{0}+b_{s},\dots b_{n}+b_{s}\right)}.$
\end{itemize}

\end{thm}

The point of Theorem \ref{dualEq} is that the dual likelihood equations define a homogeneous ideal in the polynomial ring $\quad$
$\mathbb{C}[b_0,b_1,\dots,b_n,b_s]$ whose variety is $\cBu$, the set of critical points of $\luB$ on $X'^{*}$.
Theorem \ref{dualEq} can be used to determine the ML degree of $X$ because 
$\#L_X(u)=\#\cBu$.

Since Theorem \ref{dualEq} is constructive, we express it below as an algorithm.

\begin{algorithm}\label{alg1}
Suppose $X^{*}$  in $\PP^n$ has codimension $c$. 
\begin{itemize}
\item \textbf{Input:}  Polynomials $g_1(q),g_2(q),\dots,g_l(q)$ defining $X^{*},$
and a vector $u\in \N^{n+1}$.
\item \textbf{Output:} The ML degree of $X$.
\item \textbf{Procedure:} 

Step 1. Let $G_q$ be the ideal generated by $g_1(q),\dots,g_l(q)$, and let $G_b$ be the ideal 
obtained by substituting $q_0,\dots,q_n$ for $b_0+b_s,\dots,b_n+b_s$, respectively, in the 
ideal~$G_q$. 

Step 2. Let $M_{b,u}$ denote the $(c+1)$-minors of \eqref{rankDefDual}.

Step 3. Let $S_b$ be the ideal 
 generated by the $c\times c$ minors of  
$\left.\Jac\left(X^{*}\right)\right|_{\left(b_{0}+b_{s},\dots b_{n}+b_{s}\right)}.$

Step 4. Let $W_{b,u}$ be the saturation 
$$(M_{b,u}+G_b):(b_0b_1\cdots b_nb_s\cdot S_b)^\infty$$

Step 5. Return the degree of $W_{b,u}$. 

\end{itemize}
\end{algorithm}

\begin{example}
Let $X$ be defined by $f(p)=4p_{0}p_{2}-p_{1}^2$ in $\PP^{2}$.
Then $X^{*}$ is defined by $g(q)=q_{0}q_{2}-q_{1}^2$ in the $\PP^2$.
So
\[
f(p)=\det \left[\begin{array}{cc}
2p_{0} & p_{1}\\
p_{1} & 2p_{2}
\end{array}\right]
\text{ and }
g(q)=\det \left[\begin{array}{cc}
q_{0} & q_{1}\\
q_{1} & q_{2}
\end{array}\right].
\]
The ML degree of $X$ is  computed by taking the ideal generated by 
\begin{itemize}
\item 
$g(b_0+b_s,b_1+b_s,b_2+b_s)=
(b_{0}+b_s)(b_{2}+b_s)-(b_{1}+b_s)^2,\text{ and }$
\item
 $2\times 2$ minors of 
$$\left[\begin{array}{ccc}
\frac{u_{0}}{b_{0}} & \frac{u_{1}}{b_{1}} & \frac{u_{2}}{b_{2}} \\
(b_{2}+b_s) & -2(b_{1}+b_s) & (b_{0}+b_s)
\end{array}\right]\left[\begin{array}{ccc}
b_{0}\\
 & b_1\\
 &  & b_{2}
\end{array}\right]$$
\end{itemize}
 and saturating by the product of two ideals 
\begin{itemize}
\item
the principal ideal $(b_0b_1b_2b_s)$
\text{ and }
\item 
the $1\times 1$ minors of 
$$\left[\begin{array}{ccc}
(b_{2}+b_s) & -2(b_{1}+b_s) & (b_{0}+b_s)
\end{array}\right].$$
\end{itemize}
We find  that there is a unique critical point of $l'_u(b)$ on $X'^{*}$  whose coordinates can be derived from the matrix equality
$$\frac{1}{b_{s}}\left[\begin{array}{cc}
b_{0} & b_{1}\\
b_{1} & b_{2}
\end{array}\right]=\left[\begin{array}{cc}
\frac{4u_0u_+}{(2u_0+u_1)^2} & \frac{4u_1u_+}{2(u_1+2u_2)(2u_0+u_1)}\\
 \frac{4u_1u_+}{2(u_1+2u_2)(2u_0+u_1)} &\frac{4u_2u_+}{(2u_2+u_1)^2}
\end{array}\right].$$
So by Corollary \ref{cor:mlequations}, the critical point of $\lup$ on $X$ is given by 
\[
\frac{1}{p_{s}}\left[\begin{array}{cc}
2p_{0} & p_{1}\\
p_{1} & 2p_{2}
\end{array}\right]=
\frac{1}{2u_+^2}\left[\begin{array}{c}
(2u_0+u_1)\\
(u_1+2u_2)
\end{array}\right]
\left[\begin{array}{c}
(2u_0+u_1)\\
(u_1+2u_2)
\end{array}\right]^T.
\]
\end{example}

\subsection{Experimental Results}
In this section, we compare the standard formulation of solving the likelihood equations, Algorithm 6 of \cite{HKS05}, to the dual formulation presented here, Algorithm \ref{alg1}.
All computations in this subsection  were done on a 2.8 GHz Intel Core i7 MacBook Pro using \texttt{Macaulay2}. 
$$
\begin{array}{ccl}
I_1&=&\langle 
q_0^2+2q_1^2+3q_2^2+5q_3^2
\rangle \\
I_2&=&\langle 
{q_2^2-q_1q_3, q_1q_2-q_0q_3, q_1^2-q_0q_2}
\rangle \\
{I_3}&=&\langle
q_0^3+q_1^3+q_2^3+q_3^3
\rangle \\
I_4&=&\langle 
30q_0^2+15q_1^2+10q_2^2+6q_3^2
\rangle\\
I_5&=&\langle
{q_1^2q_2^2-4q_0q_2^3-4q_1^3q_3+18q_0q_1q_2q_3-27q_0^2q_3^2}
\rangle\\
I_6&=&\langle q_0^4+q_1q_2q_3^2-q_4^4\rangle\\
I_7&=&2\times2 \minors\text{ of }
\left[ \begin{array}{c}
    2q_{11},q_{12},q_{13} \\
    q_{12},2q_{22},q_{23}\\
    q_{13},q_{23},2q_{33}
    \end{array}\right]        \\
I_8&=&2\times2 \minors\text{ of }
\left[ \begin{array}{c}
    q_{11},q_{12},q_{13} \\
    q_{12},q_{22},q_{23}\\
    q_{13},q_{23},q_{33}
    \end{array}\right]        \\
I_{10}&=&\left\langle \det
\left[ \begin{array}{c}
    q_{0},q_{1},q_{2} \\
    q_{1},q_{2},q_{3}\\
    q_{2},q_{3},q_{4}
    \end{array}\right]       
\right\rangle
\end{array}
$$

The second column in the table below is a list of ML degrees of varieties whose dual variety is given by the first column.  
The third column is the time (in seconds) it takes to calculate the ML degree using the standard formulation, while the fourth column is the time (in seconds) it takes to calculate the ML degree using the dual likelihood equations. 

\[
\begin{array}{ccrr}
X^{\star} & \text{ML degree}&  \text{Standard} & \text{Dual } \\
         I_1 & 14                    & 0.008            & 0.047     \\
         I_2 &   4                    &  0.062             & 0.062  \\  	 
	\textbf{I}_3  &   \textbf{57}			 & \textbf{166.872}	         & \textbf{1.447}    \\
         I_4 & 14                    & 0.038             & 0.042  \\  
         I_5 &  3                   & 0.017             & 0.311     \\
         I_6 & 22			&32.0657		 	&13.594\\
         I_7 & 13 			&  4.808		 	&33.489 	\\
         I_8 & 6 			& 2.349		 	&13.842\\ 
         I_{10} & 3		& 1.0731	 		&2.165
\end{array}
\]

The most notable discrepancy is in  row $3$ in bold. In this case, the ideal of $X^{*}$  is generated by a cubic, but $X$ is generated by a degree $12$ polynomial with $35$ terms.

\subsection{Tensors}
To calculate new ML degrees when $X^{*}$ is not a complete intersection  [Computation \ref{computation223}], we will work with an adjusted formulation of the dual likelihood equation. This formulation introduces codimension $X^{*}$ auxiliary unknowns (Lagrange multipliers).  Also, instead of working with every generator of the ideal of $X^{*}$, we work with  $\codim(X^{*})$ generators. 
These generators should be  chosen so that they define a reducible variety whose only irreducible component not contained in the coordinate hyperplanes is $X^{*}$.

\begin{example}\label{222tensor}
Consider $2\times 2\times 2$-tensors of the form $[p_{ijk}]$ with $i,j,k,\in \{0,1\}$. 
If $X$ is  the hyperdeterminant of these tensors, then 
$X^{*}$ is defined by the  $2\times 2$ minors of all flattenings of the tensor $[q_{ijk}]$.  
The codimension of $X^{*}$ is $4$. The $4$ flattenings below define $X^{*}$ after saturating by $q_{111}$. 
\[
\begin{array}{cc}
g_{1}\left(q\right)=q_{011}q_{101}-q_{001}q_{111} & g_{2}\left(q\right)=q_{011}q_{110}-q_{010}q_{111}\\
g_{3}\left(q\right)=q_{001}q_{110}-q_{000}q_{111} & g_{4}\left(q\right)=q_{011}q_{101}-q_{001}q_{111}
\end{array}
\]
So by introducing auxiliary unknowns $\lambda_0,\lambda_1,\dots,\lambda_4$ we create a square system of  $12$ equations  in the homogeneous variable groups $(b_{000},\dots,b_{111},b_s)$ and $(\lambda_0,\dots,\lambda_4)$. 
\[
\begin{array}{c}
g_{1}=(b_{011}+b_s)(b_{101}+b_s)-(b_{001}+b_s)(b_{111}+b_s)\\
 g_{2}=(b_{011}+b_s)(b_{110}+b_s)-(b_{010}+b_s)(b_{111}+b_s)\\
g_{3}=(b_{001}+b_s)(b_{110}+b_s)-(b_{000}+b_s)(b_{111}+b_s) \\
 g_{4}=(b_{011}+b_s)(b_{101}+b_s)-(b_{001}+b_s)(b_{111}+b_s)
\end{array}
\]
\[
[\lambda_0,\lambda_1,\lambda_2,\dots,\lambda_4]
\left[
\begin{array}{c}
\nabla\luB\\
\Jac(g)
\end{array}
\right]
\left[\begin{array}{ccc}
b_{000}\\
 & \ddots \\
 &  & b_{111}
\end{array}
\right]
=0.
\]
The solutions with $\lambda_0b_s\neq 0$ give the critical points. 
 We find that there are $13$ critical points of $\luB$ on $X^{*}$, agreeing with \cite{DSS09}, page 53.
\end{example}

The next example is a new computational result to determine the ML degree of a hyperdeterminant.

\begin{computation}\label{computation223}
Let $X$ denote the hyperdeterminant of $2\times 2\times3$ tensors of the form $[p_{ijk}]$ for $i\in\{0,1\}$, $j\in\{0,1\}$, $k\in\{0,1,2\}$.
Then the ML degree of $X$ is $71$. 
\end{computation}
\begin{proof}
The variety $X$ is dual to the variety $X^{*}$  defined by the $2\times 2$-minors of the flattenings of the  $2\times 2\times 3$ tensor $[q_{ijk}]$ with $i\in\{0,1\}$, $j\in\{0,1\}$, $k\in\{0,1,2\}$.
The variety $X^{*}$ has codimension $7$, degree $12$, and $24$ generators. 
We consider $7$ of the $24$ generators, 
\[
\begin{array}{cccccc}
g_{1}\left(q\right) & = & q_{102}q_{111}-q_{101}q_{112}\\
g_{2}\left(q\right) & = & q_{102}q_{110}-q_{100}q_{112}\\
g_{3}\left(q\right) & = & q_{002}q_{111}-q_{001}q_{112}\\
g_{4}\left(q\right) & = & q_{012}q_{102}-q_{002}q_{112}\\
g_{5}\left(q\right) & = & q_{012}q_{111}-q_{011}q_{112}\\
g_{6}\left(q\right) & = & q_{012}q_{110}-q_{010}q_{112}\\
g_{7}\left(q\right) & = & q_{002}q_{110}-q_{000}q_{112}
\end{array}
\]
such that when saturated by $q_{112}$ we recover the dual variety $X^{*}$.
We solve the following square system of equations: the seven equations
\[
g_1(b_0+b_s,\dots,b_n+b_s)=\dots=g_7(b_0+b_s,\dots,b_n+b_s)=0
\]
and the $12$ equations 
\[
[1,\lambda_1,\lambda_2,\dots,\lambda_7]
M
\left[\begin{array}{ccc}
b_{101}\\
 & \ddots \\
 &  & b_{112}
\end{array}
\right]
=0,
\]
 with $M$ being  \eqref{MM} where  ${(q_{101},\dots,{q_{112}})}={(b_{101}+b_s,\dots,{b_{112}}+b_s)}$ and $u$   consisting of random complex numbers (random  positive integers) to determine the ML degree of $X$ numerically (symbolically). 

\small\begin{equation}\label{MM}
\left[\begin{array}{cccccccccccc}
\frac{u_{101}}{b_{101}} &
\frac{u_{011}}{b_{011}} &
\frac{u_{100}}{b_{100}} &
\frac{u_{010}}{b_{010}} &
\frac{u_{001}}{b_{001}} &
\frac{u_{000}}{b_{000}} &
\frac{u_{002}}{b_{002}} &
\frac{u_{012}}{b_{012}} &
\frac{u_{102}}{b_{102}} &
\frac{u_{110}}{b_{110}} &
\frac{u_{111}}{b_{111}} &
\frac{u_{112}}{b_{112}} \\
-q_{112} & 0 & 0 & 0 & 0 & 0 & 0 & 0 & q_{111} & 0 & q_{102} & -q_{101}\\
 & -q_{112} & 0 & 0 & 0 & 0 & 0 & q_{111} & 0 & 0 & q_{012} & -q_{011}\\
 &  & -q_{112} & 0 & 0 & 0 & 0 & 0 & q_{110} & q_{102} & 0 & -q_{100}\\
 &  &  & -q_{112} & 0 & 0 & 0 & q_{110} & 0 & q_{012} & 0 & -q{}_{010}\\
 &  &  &  & -q_{112} & 0 & q_{111} & 0 & 0 & 0 & q_{002} & -q_{001}\\
 &  &  &  &  & -q_{112} & q_{110} & 0 & 0 & q_{002} & 0 & -q_{000}\\
 &  &  &  &  &  & -q_{112} & q_{102} & q_{012} & 0 & 0 & -q_{002}
\end{array}\right]
\end{equation}

The ML degree $71$ was obtained using exact methods in  \texttt{Macaulay2} in 10,943 seconds and  using 
numerical methods in \texttt{Bertini} in 5,796 seconds.
Both computations were performed on the UC Berkeley server \texttt{apppsa} which has four  16-core 2.3GHz AMD Opteron 6276 CPUs. 
The Bertini computation was done in parallel using 20 of the 64~cores.

One could have attempted to compute the number $71$ using Algorithm 6 of \cite{HKS05}.
However, to do so, we must have the defining equations of $X$. 
 We were not able to compute these equations  ourselves,
but the hyperdeterminant of $2\times 2\times 3$ tensors is 
listed on page 7 of  \cite{Bre11}. 
This is a degree 
6 polynomial with 66 terms. 
We were unable to determine the $71$ with the standard likelihood equations and with the Lagrange likelihood equations (page 4 of \cite{GR14}).   
\end{proof}

The next interesting case is when 
$X$ is the hyperdeterminant of $2\times 2 \times 2\times 2$ tensors. 
In this case, $X$ is defined by a polynomial of degree $24$ in $16$ unknowns that has $2,894,276$ terms \cite{HSYY08}.
There is no way we will be able to effectively write down the standard likelihood equations for $X$. 
However, it's dual variety $X^{*}$ is a binomial ideal consisting of the $2\times 2$-minors of all of its flattenings, and 
we may have a chance of solving the dual likelihood equations both numerically and symbolically.

\section{The Dual MLE Problem vs\\ Maximum Likelihood Duality}
In this section we introduce an example and show how the results  presented in this paper fit in context with previous work on Maximum Likelihood Duality. 
In \cite{DR13,HRS13} the notion of Maximum likelihood duality (ML-duality) was introduced.   ML-duality gave a bijection between critical points of $\lup$ on two different statistical models. 
\begin{defn}
A pair of algebraic statistical models $X$ and $Y$ in $\PP^{n}$ are said to be \textit{ML-dual} if for generic  $u$ there is an involutive bijection between  points of $L_X(u)$ and 
points of $L_Y(u)$. 
Moreover, this bijection 
pairs points of $L_X(u)$ with  points of $L_Y(u)$ such that the coordinate-wise product of each pair can be  expressed in  terms of the data $u$ alone.
\end{defn}


\begin{example}
Suppose $r\leq m\leq n$, and  let  $V_{m,n,r}$  denote the Zariski closure in $\PP^{mn-1}$ of rank $r$ matrices of the form 
$$
\left[\begin{array}{cccc}
p_{11} &p_{12}& \dots & p_{1n}\\
p_{21} &p_{22} & &\\
\vdots & &\ddots &\\
p_{m1}& & &p_{mn}
\end{array}\right].$$
Then $V_{m,n,r}^{*}$ is known to be the Zariski closure in $\PP^{mn-1}$ of rank $m-r$ matrices of the~form
$$
\left[\begin{array}{cccc}
q_{11} &q_{12}& \dots & q_{1n}\\
q_{21} &q_{22} & &\\
\vdots & &\ddots &\\
q_{m1}& & &q_{mn}
\end{array}\right].$$
Fix a choice of $m,n,r$. If we  take $X=V_{m,n,r}$, then points in $X'$ will be represented as 
$$[p_{ij}:p_s]\in X'\subset\PP^{mn}$$
and points in $X'^{*}$ will be  represented as 
$$[b_{ij}:b_s]\in X'^{*}\subset\PP^{mn}.$$

With Corollary \ref{bijectionPB}, it follows there is a bijection between $\cPu$ and $\cBu$.
On the other hand, by Theorem 1 in \cite{DR13} we know that $V_{m,n,r}$ and $V_{m,n,m-r}$ are ML-dual. This means if we take $Y$ to be $V_{m,n,m-r}$ there is an involutive bijection between critical points $L_X(u)$ and $L_Y(u)$ for generic choices of $u$.
In particular, the bijection is such that the coordinate-wise product of the paired points is  
$$\left(
\left[  \frac{u_{i+}u_{+j}u_{ij}}{u_{++}^3}:1   \right],
\left[  \frac{u_{ij}u_{++}}{u_{i+}u_{+j}}:1   \right]
\right)\in\PP^{mn}\times\PP^{mn}.
$$
Here $u_{++}:=\sum_{i,j}u_{ij}$, $u_{i+}:=\sum_j u_{ij}$, and $u_{+j}:=\sum_i u_{ij}$, and likewise for $p_{++},p_{i+},p_{+j}$. 

\end{example}

\section{Conclusion}
In this paper we have given an elegant formulation of the MLE problem involving conormal varieties.  This formulation allows us to avoid the  computation of saturation by the product of unknowns. 
We also define the dual likelihood equations that allow us  to compute critical points on $X$ even if we only know the defining equations of its dual~$X^{*}$ [Algorithm \ref{alg1}]. 
The important feature of the dual likelihood equations comes from the fact that the defining equations of $X$ may be too difficult to work with.
In addition, we showed that if we solve the dual equations, 
we can recover the critical points on $X$  [Theorem \ref{thm:projMLD}].
More broadly, we showed that if  there is a bijection between 
critical points of a function restricted to  a variety and  critical points of a monomial restricted to a different variety, then  we can formulate a  new set of ``dual" equations to determine these~points.

\section*{Acknowledgements}
The author would like to thank Elizabeth Gross, Kim Laine, Zvi Rosen, and Bernd Sturmfels for their comments and suggestions to improve earlier versions of the paper. 

\bibliographystyle{abbrv}

\end{document}